\documentclass[11pt]{article}

\pdfoutput=1

\usepackage{a4}
\usepackage[top=30truemm,bottom=30truemm]{geometry}

\usepackage[running]{lineno}

\usepackage{authblk}
\usepackage{amsmath}
\usepackage{amsthm}
\usepackage{cases}
\usepackage{booktabs}
\usepackage{microtype}
\usepackage{multirow}
\usepackage{graphicx}[dvipdfmx]
\usepackage[T1]{fontenc}

\usepackage[whole]{bxcjkjatype}

\newtheorem{theorem}{Theorem}
\newtheorem{lemma}{Lemma}
\newtheorem{corollary}{Corollary}

\theoremstyle{definition}

\newtheorem{myexample}{Example}

\graphicspath{{img/}}

\newcommand{\occ}{\mathit{occ}}

\newcommand{\MAW}{\mathsf{MAW}}

\newcommand{\typeone}{\mathcal{M}_1}
\newcommand{\typetwo}{\mathcal{M}_2}
\newcommand{\typethree}{\mathcal{M}_3}

\newcommand{\BSigma}{\Sigma_2}

\usepackage[final,mode=multiuser]{fixme}
\fxuseenvlayout{colorsig}
\fxusetargetlayout{color}
\FXRegisterAuthor{tm}{atm}{TM}
\FXRegisterAuthor{yk}{ayk}{YK}
\FXRegisterAuthor{ta}{ata}{TA}
\FXRegisterAuthor{yf}{ayf}{YF}
\FXRegisterAuthor{yn}{ayn}{YN}
\FXRegisterAuthor{si}{asi}{SI}
\FXRegisterAuthor{hb}{ahb}{HB}
\FXRegisterAuthor{mt}{amt}{MT}
\fxsetface{env}{}
\theoremstyle{definition}

\title{Combinatorics of minimal absent words \\ for a sliding window}

\author{Tooru~Akagi$^{1}$}
\author{Yuki~Kuhara$^{1}$}
\author{Takuya~Mieno$^{1,2}$}
\author{Yuto~Nakashima$^{1}$}
\author{Shunsuke~Inenaga$^{1,3}$}
\author{Hideo~Bannai$^4$}
\author{Masayuki~Takeda$^{1}$}

\affil{
  \normalsize{
  \textit{$^1$Department of Informatics, Kyushu University, Japan}\\
  \texttt{\{toru.akagi@inf.kyushu-u.ac.jp, takuya.mieno@inf.kyushu-u.ac.jp, yuto.nakashima, inenaga, takeda\}@inf.kyushu-u.ac.jp}\\
  \textit{$^2$Japan Society for the Promotion of Science, Japan}\\
  \textit{$^3$PRESTO, Japan Science and Technology Agency, Japan}\\
  \textit{$^4$M\&D Data Science Center, Tokyo Medical and Dental University, Japan}\\
  \texttt{hdbn.dsc@tmd.ac.jp}
  }
}

\date{}

\begin{document}
\maketitle

\begin{abstract}
  A string $w$ is called a minimal absent word~(MAW) for
  another string $T$
  if $w$ does not occur in $T$ but the proper substrings of $w$ occur in $T$.
  For example, let $\Sigma = \{\mathtt{a, b, c}\}$ be the alphabet.
  Then, the set of MAWs for string $w = \mathtt{abaab}$ is $\{\mathtt{aaa, aaba, bab, bb, c}\}$.
  In this paper,
  we study combinatorial properties of
  MAWs in the sliding window model, namely, 
  how the set of MAWs changes when a sliding window of fixed length $d$
  is shifted over the input string $T$ of length $n$,
  where $1 \leq d < n$.
  We present \emph{tight} upper and lower bounds on the maximum number of changes in the set of MAWs for a sliding window over $T$,
  both in the cases of general alphabets and binary alphabets.
  Our bounds improve on the previously known best bounds [Crochemore et al., 2020].
\end{abstract}

\section{Introduction}

We say that a string $s$ \emph{occurs} in another string $T$
if $s$ is a substring of $T$.
A non-empty string $w$ is said to be
a \emph{minimal absent word} (an \emph{MAW}) for
a string $T$ if $w$ does not occur in $T$
but all proper substrings of $w$ occurs in $T$.
Note that by definition
a string of length $1$ (namely a character)
which does not occur in $T$ is also an MAW for $T$.
On the other hand, any MAW for $T$ of length at least $2$
can be represented as $aub$,
where $a$ and $b$ are single characters
and $u$ is a (possibly empty) string,
such that both $au$ and $ub$ occur in $T$.
\sinote*{added}{%
  For example, let $\Sigma = \{\mathtt{a, b, c}\}$ be the alphabet.
  Then, the set of MAWs for string $w = \mathtt{abaab}$ is $\{\mathtt{aaa, aaba, bab, bb, c}\}$.
}%

\tanote*{3 citation added}{%
Applications of (minimal) absent words include
phylogeny~\cite{Chairungsee2012PhylogenyByMAW},
data compression~\cite{Crochemore2000DCA,crochemore2002improved},
musical information retrieval~\cite{CrawfordB018},
and bioinformatics~\cite{Almirantis2017MolecularBiology,Charalampopoulos18,pratas2020persistent,koulouras2021significant}.
}

\subsection{Algorithms for finding MAWs for string}

Given the above-mentioned motivations,
finding MAWs from a given string has been an important and interesting
string algorithmic problem and several nice solutions have been proposed.
The first non-trivial algorithm, which was given by
Crochemore et al.~\cite{Crochemore1998MAWdefinition},
finds the set $\MAW(T)$ of all MAWs for a given string $T$ of length $n$
over an alphabet of size $\sigma$
in $\Theta(\sigma n)$ time with $O(n)$ working space.
Since $|\MAW(T)| = O(\sigma n)$ for any string $T$ of length $n$ and
$|\MAW(S)| = \Omega(\sigma n)$ for some string $S$ of length $n$~\cite{Crochemore1998MAWdefinition},
Crochemore et al.'s algorithm~\cite{Crochemore1998MAWdefinition} runs in optimal time in the worst case.
Fujishige et al.~\cite{Fujishige2016DAWG}
improved Crochemore et al.'s algorithm so that
$\MAW(T)$ can be computed in output-sensitive $O(n+|\MAW(T)|)$ time
with $O(n)$ working space.
Both of these algorithms use the
\emph{directed acyclic word graph} (\emph{DAWG})~\cite{BlumerBHECS85} of string $T$ as a powerful tool for enumerating all MAWs for $T$.
Charalampopoulos et al.~\cite{charalampopoulos2018extended} presented an algorithm that computes all MAWs in output-sensitive time without using the DAWG of $T$.
Belazzougui et al.~\cite{Belazzougui2013ESA} showed that
$\MAW(T)$ can also be computed in $O(n+|\MAW(T)|)$ time,
provided that the \emph{bidirectional Burrows-Wheeler transform}
of a given string has already been computed.
Barton et al.~\cite{Barton2014MAWbySA} proposed a practical
algorithm to compute $\MAW(T)$ in $\Theta(n \sigma)$ time and working space\footnote{The original claimed bound in~\cite{Barton2014MAWbySA} is $O(n)$, however, the authors assumed that $\sigma = O(1)$.} based on the \emph{suffix array}~\cite{ManberM93} of $T$.
A parallel algorithm for computing MAWs has also been proposed~\cite{Barton2016ParallelComputationMAW}.
Fici and Gawrychowski~\cite{FiciG19} extended the notion
of MAWs to rooted/unrooted labeled trees and presented efficient algorithms
to compute them.

\subsection{MAWs for sliding window}
\label{sec:maw_slide_intro}
This paper follows the recent line of research on
\emph{MAWs for the sliding window model},
which was initiated by Crochemore et al.~\cite{CrochemoreHKMPR20}.
In this model, the goal is to compute or analyze $\MAW(T[i..i+d-1])$,
%for every window $T[i..i+d-1]$ of fixed length $d \geq 1$
%that shifts $T$ from left to right
%with increasing $i = 1, \ldots, n-d+1$. //replace
as $i$ is incremented, each time by 1, from 1 to $n-d+1$. For intuition, consider sliding a length-$d$ window on $T$ from left to right.
% For instance, consider

Crochemore et al.~\cite{CrochemoreHKMPR20} presented a
suffix-tree based algorithm that maintains the set of all MAWs
for a sliding window in $O(\sigma n)$ time using $O(\sigma d)$ working space.
Crochemore et al.~\cite{CrochemoreHKMPR20} also showed
how their algorithm can be applied to
approximate pattern matching under the \emph{length weighted index}
(\emph{LWI}) metric~\cite{Chairungsee2012PhylogenyByMAW}.

The (in)efficiency of their algorithms is heavily dependent on
combinatorial properties of MAWs for the sliding window.
In particular, Crochemore et al.~\cite{CrochemoreHKMPR20}
studied the number of MAWs to be added/deleted when the current window is shifted to the right by one character.
As was done in~\cite{CrochemoreHKMPR20},
for ease of discussion let us separately consider
\begin{itemize}
  \item adding a new character $T[i+d]$ to the current window $T[i..i+d-1]$ of length $d$ which forms $T[i..i+d]$, and
  \item deleting the leftmost character $T[i-1]$
        from the current window $T[i-1..i+d-1]$ which forms $T[i..i+d-1]$ of length $d$.
\end{itemize}
We remark that these two operations are symmetric.

Crochemore et al.~\cite{CrochemoreHKMPR20}
considered how many MAWs can change before and after
the window has been shifted by one position,
and showed that
\begin{eqnarray*}
  |\MAW(T[i..i+d]) \bigtriangleup \MAW(T[i..i+d-1])| \leq (s_{i}-s_{\alpha})(\sigma -1) + \sigma + 1, \\
  |\MAW(T[i-1..i+d-1]) \bigtriangleup \MAW(T[i..i+d-1])| \leq (p_{i}-p_{\beta})(\sigma - 1) + \sigma + 1,
\end{eqnarray*}
where $\bigtriangleup$ denotes the symmetric difference and
\begin{itemize}
  \item $s_{i}$ is the length of the longest repeating suffix of $T[i..i+d-1]$,
  \item $s_{\alpha}$ is the length of the longest suffix of $T[i..i+d-1]$ having an internal occurrence immediately followed by $\alpha = T[i+d]$,
  \item $p_{i}$ is that of the longest repeating prefix of $T[i..i+d-1]$, and
  \item $p_{\beta}$ is the length of the longest prefix of $T[i..i+d-1]$ having an internal occurrence immediately preceded by $\beta = T[i-1]$.
\end{itemize}
Since both $s_{i}-s_{\alpha}$ and $p_{i}-p_{\beta}$ can be at most $d-1$ in the worst case, the asymptotic bounds for the numbers of changes
in the set of MAWs obtained by Crochemore et al.~\cite{CrochemoreHKMPR20} are:
\begin{equation}
  \begin{array}{rcl}
    |\MAW(T[i..i+d]) \bigtriangleup \MAW(T[i..i+d-1])|     & \in & O(\sigma d), \\
    |\MAW(T[i-1..i+d-1]) \bigtriangleup \MAW(T[i..i+d-1])| & \in & O(\sigma d).
  \end{array}
  \label{eqn:Crochemore_bound}
\end{equation}
%\begin{eqnarray*}
%  |\MAW(T[i..i+d]) \setminus \MAW(T[i..i+d-1])| \in O(\sigma d), \\
%  |\MAW(T[i-1..i+d-1]) \setminus \MAW(T[i..i+d-1])| \in O(\sigma d).
%\end{eqnarray*}

Crochemore et al.~\cite{CrochemoreHKMPR20} also considered
the \emph{total changes} in the set of MAWs for every sliding window
over the string $T$, and showed that
\begin{equation}
  \sum_{i=1}^{n-d} \big( |\MAW(T[i..i+d-1]) \bigtriangleup \MAW(T[i+1..i+d])| \big)\in O(\sigma n). \label{eqn:total_Crochremore}
\end{equation}

\subsection{Our contribution}
The goal of this paper is to give more rigorous analyses
on the number of MAWs for the sliding window model.
This study is well motivated since
revealing more combinatorial insights to the sets of MAWs for the sliding windows
can lead to more efficient algorithms for computing them.

In this paper, we first give the following upper bounds:
\begin{equation}
  \begin{array}{rcl}
    |\MAW(T[i..i+d]) \bigtriangleup \MAW(T[i..i+d-1])|     & \leq & d + \sigma_{i,i+d-1} + 1, \\
    |\MAW(T[i-1..i+d-1]) \bigtriangleup \MAW(T[i..i+d-1])| & \leq & d + \sigma_{i,i+d-1} + 1,
  \end{array}
  \label{eqn:tight_bounds}
\end{equation}
where $\sigma_{x,y}$ is the number of distinct characters in $T[x,y]$.
We then show that our new upper bounds in (\ref{eqn:tight_bounds}) are \emph{tight}
by showing a family of strings achieving these bounds.

Since $\sigma_{i,i+d-1} \leq d$ always holds,
we immediately obtain new asymptotic upper bounds
\begin{equation}
  \begin{array}{rcl}
    |\MAW(T[i..i+d]) \bigtriangleup \MAW(T[i..i+d-1])|     & \in & O(d), \\
    |\MAW(T[i-1..i+d-1]) \bigtriangleup \MAW(T[i..i+d-1])| & \in & O(d).
  \end{array}
  \label{eqn:new_asymptotic_bounds}
\end{equation}
Our new upper bounds in (\ref{eqn:new_asymptotic_bounds}) improve Crochemore et al.'s upper
bounds in (\ref{eqn:Crochemore_bound})
for any alphabet of size $\sigma \in \omega(1)$.
Our upper bounds in (\ref{eqn:new_asymptotic_bounds})
are also \emph{tight} as there exists a family of strings achieving the matching lower bounds $\Omega(d)$.

In this paper, we also present
a new upper bound for the total changes of MAWs:
\begin{equation}
  \sum_{i=1}^{n-d} \big( |\MAW(T[i..i+d-1]) \bigtriangleup \MAW(T[i+1..i+d])| \big)\in O(\min\{\sigma, d\} n)
  \label{eqn:total_new}
\end{equation}
which improves the previous bound $O(\sigma n)$ in (\ref{eqn:total_Crochremore}).
We then show that this new upper bound
in (\ref{eqn:total_new}) is also \emph{tight}.

All of our new bounds aforementioned
are tight for any alphabet of size $\sigma_{i,i+d-1} \geq 3$.
We further explore the case of binary alphabets with $\sigma_{i,i+d-1} = 2$,
and show that there exist even tighter bounds in the binary case.
Namely, for $\sigma_{i,i+d-1} = 2$, we prove that
\begin{equation}
  \begin{array}{rcl}
    |\MAW(T[i..i+d]) \bigtriangleup \MAW(T[i..i+d-1])|     & \leq & \max\{3, d\}, \\
    |\MAW(T[i-1..i+d-1]) \bigtriangleup \MAW(T[i..i+d-1])| & \leq & \max\{3, d\}.
  \end{array}
  \label{eqn:binary_bounds}
\end{equation}
We remark that plugging $\sigma_{i,i+d-1} = 2$ into
(\ref{eqn:tight_bounds}) for the general case
only gives $d + \sigma_{i,i+d-1} + 1 = d + 3$,
which is larger than $\max\{3, d\}$ in (\ref{eqn:binary_bounds}).
We consider the case $\sigma_{i,i+d-1} \geq  d$ in Lemmas~\ref{lem:case_d=1} and \ref{lem:case_d=2}.
We also show that the upper bounds $\max\{3, d\}$ in (\ref{eqn:binary_bounds})
are \emph{tight} by giving matching lower bounds with a family of binary strings.
 
A part of the results reported in this article appeared in
a preliminary version of this paper~\cite{MienoKAFNIBT20}.
In addition, this present article considers the case of binary alphabets
and presents tight upper and lower bounds for this case.

\section{Preliminaries}

\subsection{Strings}
Let $\Sigma$ be an alphabet.
An element of $\Sigma$ is called a character.
An element of $\Sigma^\ast$ is called a string.
The length of a string $T$ is denoted by $|T|$.
The empty string $\varepsilon$ is the string of length 0.
If $T = xyz$, then $x$, $y$, and $z$ are called
a \emph{prefix}, \emph{substring}, and \emph{suffix} of $T$, respectively.
They are called a \emph{proper prefix}, \emph{proper substring},
and \emph{proper suffix} of $T$ if $x \neq T$, $y \neq T$, and $z \neq T$,
respectively.

For any $1 \le i \le |T|$, the $i$-th character of $T$ is denoted by $T[i]$.
For any $1 \le i \le j \le |T|$, $T[i..j]$ denotes
the substring of $T$ starting at $i$ and ending at $j$.
For convenience, let $T[i..j] = \varepsilon$ for $0 \leq j < i \leq |T|+1$.
For any $i \le |T|$ and $1 \le j$, let $T[..i] = T[1..i]$ and $T[j..] = T[j..|T|]$.
%For a string $w$, the set of beginning positions of occurrences of $w$ in $T$
%is denoted by $\occ_T(w) = \{i \mid T[i..i+|w|-1] = w\}$.
%Let $\numocc_T(w) = |\occ_T(w)|$.
%For convenience, let $\numocc_T(\varepsilon) = |T|+1$.
%
%In what follows, we consider an arbitrarily fixed string $T$ of length $n \ge 1$
%over an alphabet $\Sigma$ of size $\sigma \ge 2$.

We say that a string $w$ \emph{occurs} in a string $T$
if $w$ is a substring of $T$.
Note that by definition the empty string $\varepsilon$
is a substring of any string $T$
and hence $\varepsilon$ always occurs in $T$.

\subsection{Minimal absent words (MAWs)}
%Any string $w$ is said to be \emph{absent} from $T$
%%if $\numocc_T(w) = 0$
%and \emph{present} in $T$ if $\numocc_T(w) \geq 1$.
%We assume that $\numocc_T(\varepsilon) = |T|+1$ for any string $T$.
%For any substring $w$ of $T$,
%$w$ is called
%\emph{unique} in $T$ if $\numocc_T(w) = 1$,
%\emph{quasi-unique} in $T$ if $1 \leq \numocc_T(w) \leq 2$,
%and \emph{repeating} in $T$ if $\numocc_T(w) \ge 2$.

A string $w$ is called an \emph{absent word} for a string $T$
if $w$ does not occur in $S$.
An absent word $w$ for $S$ is called a \emph{minimal absent word}
or \emph{MAW} for $S$ if
any proper substring of $w$ occurs in $S$.
We denote by $\MAW(S)$ the set of all MAWs for $S$.
By the definition of MAWs,
it is clear that $w \in \MAW(S)$ iff the three following conditions hold:
\begin{enumerate}
  \item[(A)] $w$ does not occur in $S$;
  \item[(B)] $w[2..]$ occurs in $S$;
  \item[(C)] $w[..|w|-1]$ occurs in $S$.
\end{enumerate}
We note that if $w$ is a string of length $1$
which does not occur in $S$
(i.e. $w$ is a single character in 
\tanote*{add the number of alphabet used in T}{
$\Sigma$ of size $\sigma$}
 not occurring in $S$),
then $w$ is a MAW for $S$ since $w[2..] = w[..|w|-1] = \varepsilon$
is a substring of $S$.

%
%For example, let $T = \mathtt{abaab}$ be a string
%over an alphabet $\Sigma = \{\mathtt{a}, \mathtt{b}, \mathtt{c}\}$.
%Then $x = \mathtt{aaba}$ is a MAW for $T$, since $x[2..] = \mathtt{aba}$ and $x[..|x|-1] = \mathtt{aab}$ are
%both present in $T$ and $x$ is absent from $T$.
%Also, $\MAW(T) = \{\mathtt{aaa}, \mathtt{aaba}, \mathtt{bab}, \mathtt{bb}, \mathtt{c}\}$\footnote{Note that
%every single character in $\Sigma$ which is absent from $T$ is a MAW for $T$.}.
%

\tanote*{example added}{%
 \begin{myexample}
  Let $\Sigma = \{ \mathtt{a,b,c,d} \}. $ Then, the set of MAWs for string $\mathtt{cbaaaa}$ is:
  \[
\MAW(\mathtt{cbaaaa}) = \{\mathtt{cc,bb,aaaaa,bc,ab,ca,ac,d}\}.
  \]
  \end{myexample}

}%

\subsection{MAWs for a sliding window}

  Given a string $T$ of length $n$ and
  a sliding window $S_i = T[i..j]$ of length $d = j-i+1$
  for increasing $i = 1, \ldots, n-d+1$,
  our goal is to analyze how many MAWs for the sliding window
  can change when the window shifts over the string $T$.
  We will consider both the maximum change per one shift,
  and the maximum total number of changes when sliding the window from the beginning to the end.

  %As was already discussed in Section~\ref{sec:maw_slide_intro},
  As was done in~\cite{CrochemoreHKMPR20},
  for simplicity, we separately consider two symmetric operations
  of appending a new character to the right of the window and
  of deleting the leftmost character from the window.

  \begin{myexample}
  Let $\Sigma = \{ \mathtt{a,b,c,d} \}$. Consider appending character $\mathtt{c}$ to the right of string $\mathtt{cbaaaa}$.
  Then,
  \[
    \begin{array}{lcl}
      \MAW(\mathtt{cbaaaa})  & = & \{\mathtt{cc,bb,aaaaa,bc,ab,ca,ac,d}\},                    \\
      \MAW(\mathtt{cbaaaac}) & = & \{\mathtt{cc,bb,aaaaa,bc,ab,ca,   acb,bac,baac,baaac,d}\}.
    \end{array}
  \]
  Thus
  $\MAW(\mathtt{cbaaaa}) \bigtriangleup \MAW(\mathtt{cbaaaac}) = \{\mathtt{\underline{ac}, acb,bac,baac,baaac}\}$,
  where the underlined string is deleted from
  and the strings without underlines are added to
  the set of MAWs by appending $\mathtt{c}$ to $\mathtt{cbaaaa}$.
  \end{myexample}

%This paper deals with the problems of counting how many MAWs
%
%in a sliding window
%of fixed length $d$ over a given string $T$, formalized as follows:
%\begin{description}
%  \item[Input:] String $T$ of length $n$ and positive integer $d$~($< n$).
%  \item[Output:] $\MAW(T[i..i+d-1])$ for all $1 \leq i \leq n - d + 1$.
%\end{description}

\section{Tight bounds on the changes to MAWs for sliding window} \label{sec:comb_MAWs}

In this section,
we present our new bounds for the
changes of MAWs for the sliding window over the string $T$.
In Section~\ref{sec:append},
we consider the number of changes of MAWs when
the current window $T[i..j]$ is extended by adding a new character $T[j+1]$.
Section~\ref{sec:remove} is for the symmetric case
where the leftmost character $T[i]$ is deleted from $T[i..i+j+1]$.
Finally, in Section~\ref{sec:total},
we consider the total number of changes of MAWs
while the window has been shifted from the beginning of $T$ until its end.

\subsection{Changes to MAWs when a character is appended to the right}
\label{sec:append}
We consider the number of changes of MAWs when appending
$T[j+1]$ to the current window $T[i..j]$.

For the number of deleted MAWs, the next lemma is known:
\begin{lemma}[\cite{CrochemoreHKMPR20}] \label{lem:extention_dec}
  For any $1 \le i \le j < n$,
  $|\MAW(T[i..j])\setminus \MAW(T[i..j+1])| = 1$.
\end{lemma}

Next, we consider the number of added MAWs.
We classify the MAWs in $\MAW(T[i..j+1]) \setminus \MAW(T[i..j])$
to the following three types\footnote{At least one of $w[2..]$ and $w[..|w|-1]$ does not occur in $T[i..j]$,
  since $w \not\in \MAW(T[i..j])$.}~(see Figure~\ref{fig:MAW_3types}).
A MAW $w$ in $\MAW(T[i..j+1]) \setminus \MAW(T[i..j])$ is said to be of:
\begin{description}
  \item Type 1 if
    neither $w[2..]$ nor $w[..|w|-1]$ occurs in $T[i..j]$;
  \item Type 2 if
    $w[2..]$ occurs in $T[i..j]$ but $w[..|w|-1]$ does not occur in $T[i..j]$;
  \item Type 3 if
    $w[2..]$ does not occur in $T[i..j]$ but $w[..|w|-1]$ occurs in $T[i..j]$.
\end{description}

\begin{figure}[tb]
    \centerline{\includegraphics[width=1.0\linewidth]{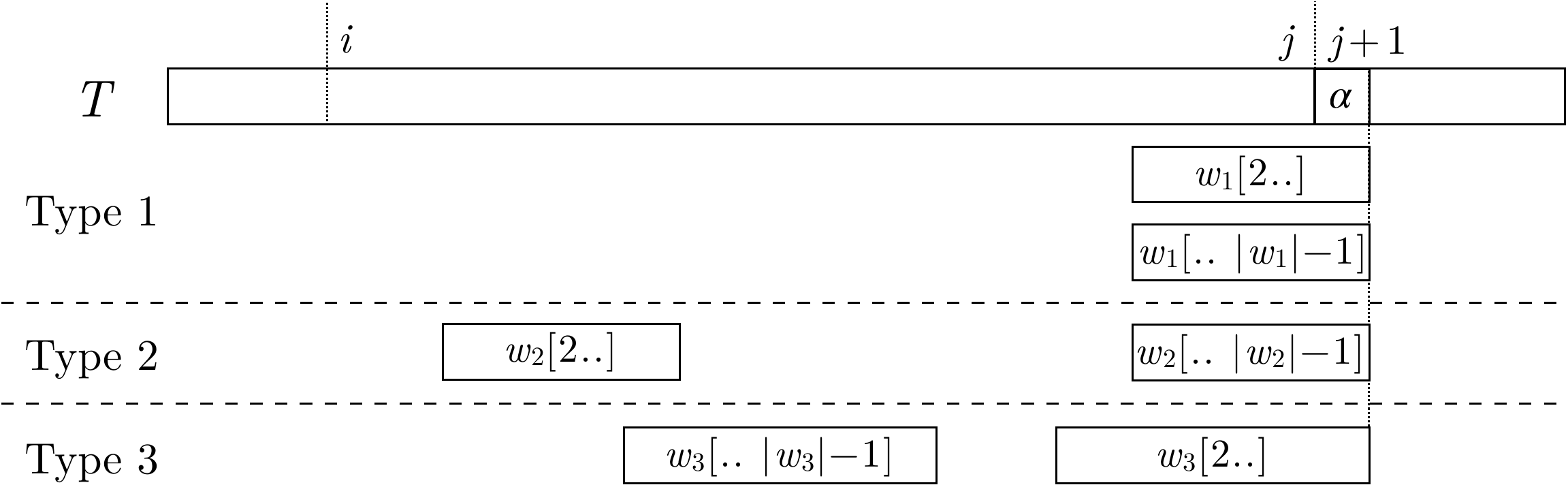}}
    \caption{
      Illustration for the three types of MAWs, where
      $w_1 \in \typeone$, $w_2 \in \typetwo$, and $w_3 \in \typethree$.
    }
    \label{fig:MAW_3types}
\end{figure}

We denote by $\typeone$, $\typetwo$, and $\typethree$ the sets of MAWs of Type 1, Type 2 and Type 3, respectively.
Recall that $w$ is a MAW for $T[i..j+1]$.

%In the sequel, we consider the numbers of MAWs of Type 1, 2, and 3
%that increase after appending $T[j+1]$ to $T[i..j]$.

Let $\sigma_{i,j}$ be the number of distinct characters occurring in
the current window $T[i..j]$.

%Regarding $\typeone$ and $\typetwo$, it is shown that $|\typeone| \le 1$ by \cite{CrochemoreHKMPR20}.
%It is also shown in \cite{CrochemoreHKMPR20} that
%  the last characters of all MAWs in $\typetwo$ are all different.
%Furthermore, by the definition of $\typetwo$,
%  the last character of each MAW for $\typetwo$ occurs in $T[i..j]$.
%  Thus, $|\typetwo| \le \sigma_{i,j}$.
\tanote*{moved from just before lemma 4 to here}{%
  The next three lemmas show the upper bounds of $\typeone$, $\typetwo$, and $\typethree$:
  }
  
  \begin{lemma}[\cite{CrochemoreHKMPR20}]
  \label{type1max}
  For any $1 \le i \le j < n$,
  $|\typeone| \le 1$.
  Also, if $\alpha$ is the character appended to $T[i..j]$,
  then the only element of $\typeone$ is of the form $\alpha^k$ for some $k \geq 1$.
  \end{lemma}
  
  \begin{lemma}
  \label{type2max}
  For any $1 \le i \le j < n$,
  $|\typetwo| \le \sigma_{i,j}$.
  \end{lemma}

  \begin{proof}
    It is shown in \cite{CrochemoreHKMPR20} that
    the last characters of all MAWs in $\typetwo$ are all distinct.
    Furthermore, by the definition of $\typetwo$,
    the last character $T[j+1]$ of each MAW in $\typetwo$ must occur in the current window $T[i..j]$.
    Thus, $|\typetwo| \le \sigma_{i,j}$.
  \end{proof}
  
  \begin{lemma}
  \label{type3max}
  For any $1 \le i \le j < n$,
  $|\typethree| \le d-1$, where
  $d = j - i + 1$.
  \end{lemma}

\begin{proof}
  We show that there is an injection $f: \typethree \rightarrow [i, j-1]$
  which maps each MAW $w \in \typethree$ to
  the ending position of the leftmost occurrence of $w[..|w|-1]$ in
  the current window $T[i..j]$.
  
  First, we show that the range of this function $f$ is $[i, j-1]$.
  By definition,
  $w$ is absent from $T[i..j+1]$ and $w[|w|] = T[j+1]$
  for each $w \in \typethree$,
  and thus, no occurrence of $w[..|w|-1]$ in $T[i..j]$
  ends at position $j$.
  Hence, the range of $f$ does not contain the position $j$, i.e. it is $[i, j-1]$.
  
  Next, for the sake of contradiction,
  we assume that $f$ is not an injection,
  i.e. there are two distinct MAWs $w_1, w_2 \in \typethree$ such that $f(w_1) = f(w_2)$.
  Without loss of generality, assume $|w_1| \ge |w_2|$.
  Since $w_1[|w_1|] = w_2[|w_2|] = T[j+1]$ and $f(w_1) = f(w_2)$, $w_2$ is a suffix of $w_1$.
  If $|w_1| = |w_2|$, then $w_1 = w_2$ and it contradicts with $w_1 \neq w_2$.
  If $|w_1| > |w_2|$, then $w_2$ is a proper suffix of $w_1$, and it contradicts
  with the fact that $w_2$ is absent from $T[i..j+1]$~(see Figure~\ref{fig:MAW_Type3}).
  Therefore, $f$ is an injection and
  $|\typethree| \le j-1-i+1 = d-1$.
\end{proof}

\begin{figure}[tb]
    \centerline{\includegraphics[width=1.0\linewidth]{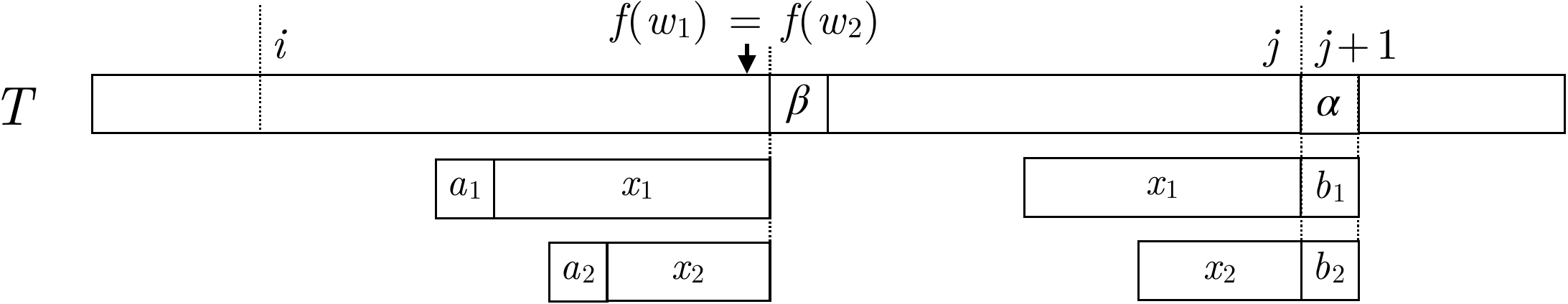}}
    \caption{
      Illustration for the contradiction in the proof of Lemma~\ref{lem:extention_inc}.
      Consider two strings $w_1 = a_1x_1b_1$ and $w_2 = a_2x_2b_2$ that are MAWs for $T$ of Type 3
      where $a_1, a_2, b_1, b_2 \in \Sigma$ and $x_1, x_2 \in \Sigma^\ast$.
      If $|w_1| > |w_2|$ and $f(w_1) = f(w_2)$, then $x_2$ is a proper suffix of $x_1$,
      and it contradicts that $a_2x_2b_2$ is absent from $T$.
    }\label{fig:MAW_Type3}
\end{figure}

Summing up all the upper bounds for
$\typeone$, $\typetwo$, and $\typethree$,
we obtain the following lemma:
\begin{lemma}\label{lem:extention_inc}
  For any $1 \le i \le j < n$,
  $|\MAW(T[i..j+1]) \setminus \MAW(T[i..j])| \le \sigma_{i,j} + d$, where
  $d = j - i + 1$.
\end{lemma}

\begin{proof}
  Immediately follows from Lemmas~\ref{type1max}, \ref{type2max}, and \ref{type3max} and that $\typeone$, $\typetwo$, and $\typethree$ are mutually disjoint.
\end{proof}

%The next lemma follows from Lemma~\ref{lem:extention_dec} and Lemma~\ref{lem:extention_inc}.
Now we obtain the main result of this subsection,
which shows the matching upper and lower bounds for
$|\MAW(T[i..j+1]) \bigtriangleup \MAW(T[i..j])|$.

\begin{theorem}\label{lem:extention_total}
  For any $1 \le i \le j < n$,
  $|\MAW(T[i..j+1]) \bigtriangleup \MAW(T[i..j])| \le \sigma_{i,j} + d + 1$, where
  $d = j-i+1$.
  The upper bound is tight
  when $\sigma \ge 3$ and $\sigma_{i,j} + 1 \le \sigma$.
\end{theorem}

\begin{proof}
  By Lemma~\ref{lem:extention_dec} and Lemma~\ref{lem:extention_inc}, we have
  %It follows from Lemma~\ref{lem:extention_dec} and Lemma~\ref{lem:extention_inc} that
  $|\MAW(T[i..j+1]) \bigtriangleup \MAW(T[i..j])| = |\MAW(T[i..j+1]) \setminus \MAW(T[i..j])| + |\MAW(T[i..j]) \setminus \MAW(T[i..j+1])|\
  \le \sigma_{i,j} + d + 1$.
  
  In the following, we show that the upper bound is tight, i.e.
  there is a string $Z$ of length $d$ and a character $\alpha$,
  where $|\MAW(Z) \bigtriangleup \MAW(Z\alpha)| = \sigma_{1,d} + d + 1$
  for any two integers $d$ and $\sigma_{1,d}$ with $1 \le \sigma_{1,d} \le d$ and $\sigma_{1,d} + 1 \le \sigma$.
  Namely, in this example, we set $i = 1$ and $j = d$.
  Let $\Sigma = \{a_1, a_2, \cdots, a_\sigma\}$ be an alphabet.
  Given two integers $d$ and $\sigma_{1,d}$
  with $1 \le \sigma_{1,d} \le d$ and $\sigma_{1,d} + 1 \le \sigma$,
  consider a string $Z = a_1a_2 \cdots a_{\sigma_{1,d}-1}a_{\sigma_{1,d}}^{d-\sigma_{1,d}+1}$ of length $d$
  and a character $\alpha = a_{\sigma_{1,d}+1}$.
  Then,
  $$\MAW(Z) \setminus \MAW(Z\alpha) = \{\alpha\}.$$
  Also, 
  \begin{eqnarray*}
    \MAW(Z\alpha) \setminus \MAW(Z) & = & 
    \{\alpha^2\} \cup 
    \{\alpha a_i\mid 1 \le i \le \sigma_{1,d} \} \cup 
    \{ a_i\alpha\mid 1 \le i \le \sigma_{1,d} -1\} \\ 
    & & \cup \{a_{\sigma_{1,d}-1}a_{\sigma_{1,d}}^e\alpha\mid 1 \le e \le d-\sigma_{1,d}\}.
  \end{eqnarray*}
  This leads to the matching lower bound $|\MAW(Z) \bigtriangleup \MAW(Z\alpha)| = \sigma_{1,d} + d + 1$.
\end{proof}

A concrete example for our lower-bound strings $Z$ and $Z\alpha$ is shown below.

\begin{myexample}
  Consider a string $Z = \mathtt{abcddd}$ with $\sigma_{1,6} = 4$, and let $d = |Z| = 6$. We have $d-\sigma_{1,6}+1 = 3$.
  Also, let $\alpha = \mathtt{e}$.
  Then,
  \[
   \MAW(\mathtt{abcddd}) \setminus \MAW(\mathtt{abcddde}) = \{\mathtt{e}\}
  \]
  and
  \begin{eqnarray*}
    \lefteqn{\MAW(\mathtt{abcddde}) \setminus \MAW(\mathtt{abcddd})} \\
    & = & \typeone \cup \typetwo \cup \typethree \\
    & = & \{\mathtt{ee}\} \cup \{\mathtt{ea, eb, ec, ed}\} \cup \{\mathtt{ae, be, ce, cde, cdde}\},
  \end{eqnarray*}
  and therefore $|\MAW(Z) \bigtriangleup \MAW(Z\alpha)| = \sigma_{1,6} + d + 1 = 11$.
\end{myexample}

\subsubsection{Changes to MAWs when a character already occurring in the window is added to the right}

In this subsection, we consider the case where a new character $T[j+1]$ that is appended to the right of the current window $T[i..j]$ already occurs in $T[i..j]$. This means that $\sigma_{i,j} = \sigma_{i,j+1}$, i.e., the alphabet size does not increase before and after the new character is added.

The next lemma shows that a conflict occurs between $\typeone$ and $\typetwo$ when $\sigma_{i,j} = \sigma_{i,j+1}$.

\tanote*{generalized lemma for the non binary case}{%
\begin{lemma}\label{MAW1and2_collide}
  For any $T[i..j]$ such that $d = j-i+1 \geq 3$ and $\sigma_{i,j} = \sigma_{i,j+1}$,
  $|\typeone|+|\typetwo| \leq \sigma_{i,j}$.
\end{lemma}

\begin{proof}
  Let $\mathtt{c} = T[j+1]$ and let $k$ be the length of the maximum run of $\mathtt{c}$'s that is a suffix of $T[i .. j]$.
  If $T[j] \neq \mathtt{c}$ then let $k = 0$.
  By the definition of $\typeone$, $\mathtt{c}^{k+2}$ is the only candidate
  for the Type-1 MAW for $T[i .. j+1]$, in which case $au = ub = \mathtt{c}^{k+1}$ occurs only once
  in $T[i .. j+1]$ as a suffix.
  This means that $\mathtt{c}^{k+2}$ can be a Type-1 MAW for $T[i .. j+1]$
  only if $\mathtt{c}^k$ is the longest run of $\mathtt{c}$'s in $T[i..j]$.

  Now suppose that $\mathtt{c}^{k+2}$ is a Type-1 MAW for $T[i .. j+1]$,
  and let $a'u'\mathtt{c}$ denote a Type-2 MAW for $T[i .. j+1]$.
  Then, by definition, $a'u'$ occurs only once in $T[i .. j+1]$ as a suffix
  (see also the middle of Figure~\ref{fig:MAW_3types}).
  \begin{itemize}
    \item If $|u'| \geq k$, then 
    $\mathtt{c}^{k+1}$ is a suffix of $u'$ as shown in Figure~\ref{binarymaw_figure2}.
    However, by the definition of Type-2 MAWs,
    $u'\mathtt{c}$ must occur in $T[i..j]$ (see also the middle of Figure~\ref{fig:MAW_3types}),
    which implies that $\mathtt{c}^{k+1}$ occurs in $T[i..j]$.
    This contradicts that $\mathtt{c}^k$ is the longest run of $\mathtt{c}$'s in $T[i..j]$.

   \item If $|u'| < k$,
     then $a'u'\mathtt{c} = \mathtt{c}^{|a'u'\mathtt{c}|}$ with $|a'u'\mathtt{c}| \leq k+1$ occurs in $T[i .. j+1]$ as a suffix,
     and this contradicts that $a'u'\mathtt{c}$ is a MAW for $T[i .. j+1]$.
  \end{itemize}
  Hence $a'u'\mathtt{c}$ cannot be in $\typetwo$, which leads to $|\typetwo| \leq \sigma_{i,j}-1$ by Lemma \ref{type2max}.
  Thus, $|\typeone|+|\typetwo| \leq \sigma_{i,j}$
  for any string $T[i .. j+1]$ such that $T[i..j]$ contains at least one character that is equal to $T[j+1]$.
\end{proof}

\begin{figure}[tb]
 \centering
 \includegraphics[keepaspectratio, scale=0.35]
      {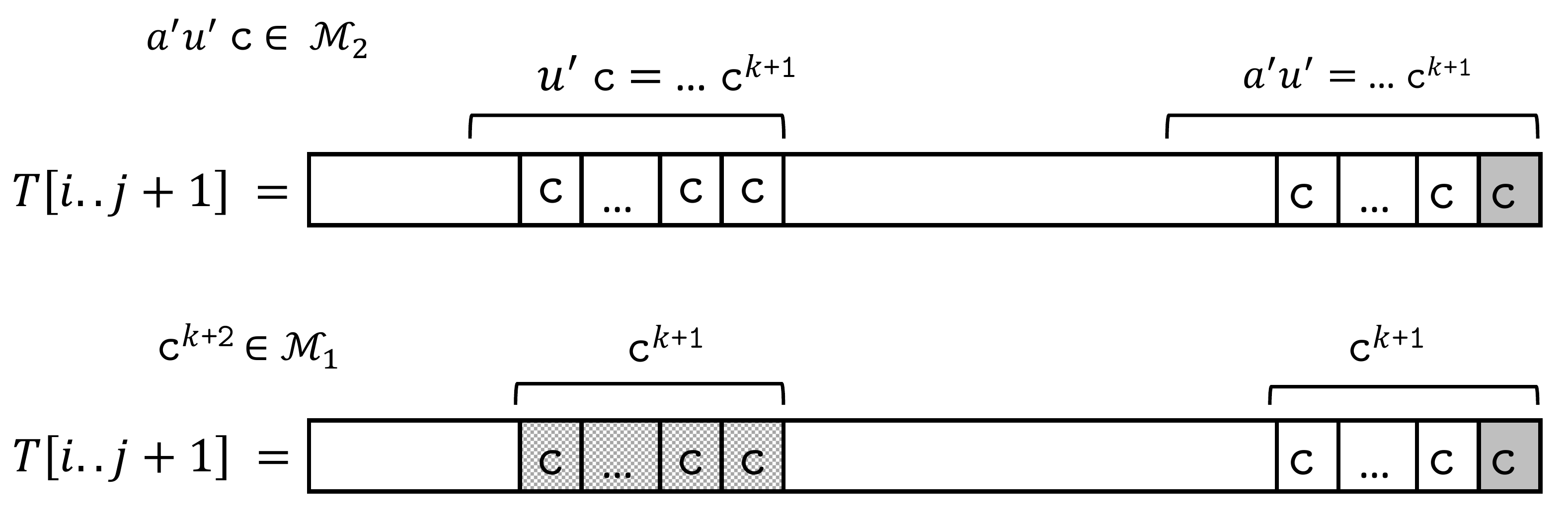}
 \caption{Collision between the new Type-2 MAW and Type-1 MAW, where the rightmost $\mathtt{c}$ in gray is the new appended character in each picture.}
 \label{binarymaw_figure2}
\end{figure}

Recall that Lemma~\ref{type1max} and Lemma~\ref{type2max} in the case where $\sigma_{i,j+1} \geq \sigma_{i,j}$
gives us $|\typeone|+|\typetwo| = \sigma_{i,j}+1$.
Compared to this, Lemma~\ref{MAW1and2_collide} shaves the total size of $\typeone$ and $\typetwo$ by one
in the case where $T[j+1]$ already occurs in $T[i..j]$.
Coupled with Lemma~\ref{type3max}, Lemma~\ref{MAW1and2_collide} leads us to the following corollary:

\begin{corollary}\label{col:extention_total}
  For any $1 \le i \le j < n$,
  $|\MAW(T[i..j+1]) \bigtriangleup \MAW(T[i..j])| \le \sigma_{i,j+1} + d$, where
  $d = j-i+1$.  %The upper bound is tight when $\sigma \ge 3$. % and $\sigma_{i,j} = \sigma$.
\end{corollary}

  }

\subsection{Changes to MAWs when the leftmost character is deleted}
\label{sec:remove}
Next, we analyze the number of changes of MAWs when deleting the leftmost character from a string.
By a symmetric argument to Theorem~\ref{lem:extention_total}, we obtain:

\begin{corollary}\label{lem:reduction_total}
  For any $1 < i \le j \le n$,
  $|\MAW(T[i.. j]) \bigtriangleup \MAW(T[i-1.. j])| \le \sigma_{i,j} + d + 1$ where
  $d = j-i+1$ and
  $\sigma_{i,j}$ is the number of distinct characters that occur in $T[i.. j]$.
  Also, the upper bound is tight
  when $\sigma \ge 3$ and $\sigma_{i,j} + 1 \le \sigma$.
\end{corollary}

%\begin{proof}
%  Symmetric to the proof of Lemma~\ref{lem:extention_total}.
%\end{proof}
%

%
Finally, by combining Theorem~\ref{lem:extention_total} and Corollary~\ref{lem:reduction_total},
we obtain the next theorem:
\begin{theorem}\label{col:increase_MAW_at_one_slide}
  Let $d$ be the window length.
  For any string $T$ of length $n > d$ and each position
  $i$ in $T$ with $1 \le i \le n-d$,
  $|\MAW(T[i..i+d-1]) \bigtriangleup \MAW(T[i+1..i+d])| \in O(d)$.
  Also, there exists a string $T'$ with $|T'| \geq d+1$
  which satisfies
  $|\MAW(T'[j..j+d-1]) \bigtriangleup \MAW(T'[j+1..j+d])| \in \Omega(d)$
  for some $j$ with $1 \le j \le |T'|-d$.
\end{theorem}

This theorem improves Crochemore et al.'s upper bound
for $|\MAW(T[i..i+d-1]) \bigtriangleup \MAW(T[i+1..i+d])| \in O(\sigma d)$
for any alphabet of size $\sigma \in \omega(1)$.

\subsection{Total changes of MAWs when sliding a window on a string}
\label{sec:total}
In this subsection, we consider the total number of changes of MAWs
when sliding the window of length $d$ from the beginning of $T$ to the end of $T$.
We denote the total number of changes of MAWs by
$\mathcal{S}(T, d) = \sum_{i=1}^{n-d}|\MAW(T[i..i+d-1]) \bigtriangleup \MAW(T[i+1..i+d])|$.
The following lemma is known:
\begin{lemma}[\cite{CrochemoreHKMPR20}] \label{lem:sigma_slide_MAW_Crochemore}
  For a string $T$ of length $n > d$ over an alphabet $\Sigma$ of size $\sigma$,
  $\mathcal{S}(T, d) \in O(\sigma n)$.
\end{lemma}

The aim of this subsection is to give a more rigorous bound for $\mathcal{S}(T, d)$.
We first show that the above bound is tight under some conditions.
\begin{lemma}\label{lem:sigma_slide_MAW}
  The upper bound of Lemma~\ref{lem:sigma_slide_MAW_Crochemore}
  is tight when $\sigma \le d$ and $n-d \in \Omega(n)$.
\end{lemma}

\begin{proof}
%  The upper bound $\mathcal{S}(T, d) \in O(\sigma n)$ is proved in \cite{CrochemoreHKMPR20}.
%  Next, we show that the upper bound is tight if $\sigma \le d$ and $n-d \in \Omega(n)$.
\tanote*{updated the binary case in detail}{%
  If $\sigma = 2$, the lower bound
  $\mathcal{S}(T, d) \in \Omega(n-d) = \Omega(\sigma (n-d))$ is obtained
  by string $T = (\mathtt{ab})^{n/2}$ since 
$\MAW((\mathtt{ab})^{d/2}) \bigtriangleup \MAW((\mathtt{ba})^{d/2}) = \{ (\mathtt{ab})^{d/2} ,(\mathtt{ba})^{d/2}\}$.
}
  In the sequel, we consider the case where $\sigma \ge 3$.
  Let $k$ be the integer with $(k-1)(\sigma-1) \le d < k(\sigma-1)$.
  Note that $k \geq 2$ since $\sigma \le d$.
  Let $\Sigma = \{a_1, a_2, \cdots, a_\sigma\}$ and $\alpha = a_\sigma$.
  We consider a string $T' = U^e + U[..m]$ where
  $U = a_1 \alpha^{k-1} a_2 \alpha^{k-1} \dots a_{\sigma-1} \alpha^{k-1}$,
  $e = \lfloor \frac{n}{k(\sigma-1)} \rfloor$, and $m = n\bmod k(\sigma-1)$.
  Let $c$ be a character that is not equal to $\alpha$.
  For any two distinct occurrences $i_1, i_2 \in \occ_{T'}(c)$ for $c$,
  $|i_1 - i_2| \ge k(\sigma-1) > d$.
  Thus, any character $c \ne \alpha$ is absent from at least one of two adjacent windows
  $T'[i..i+d-1]$ and $T'[i+1..i+d]$ for every $1 \le i \le n-d$.

  Now we consider a window $W = T'[p-d..p-1]$ where
  $d + 1 \le p \le n$ and $T'[p] = \beta \ne \alpha$.
  Let $\Pi = \{b_1, b_2, \cdots, b_{\pi-1}, \alpha\} \subset \Sigma \setminus \{\beta\}$
  be the set of all $\pi$ characters that occur in $W$.
  Without loss of generality, we assume that the current window is $W = \alpha^r b_1\alpha^{k-1}b_2\alpha^{k-1}\cdots b_{\pi-1}\alpha^{k-1}$
  and the next window is $W' = W[2..]\beta$
  where $r = d\bmod k$~(see Figure~\ref{fig:MAW_total}).
  For any character $b \in \Pi \setminus \{b_1, b_{\pi-1}, \alpha\}$,
  $b \alpha^{\ell} \beta$ is in $\MAW(W')\bigtriangleup\MAW(W)$ for every $0 \le \ell \le k-1$.
  If $r > 0$,
  $b_1 \alpha^{\ell} \beta$ is also in $\MAW(W')\bigtriangleup\MAW(W)$ for every $0 \le \ell \le k-1$.
  Otherwise,
  $b_1$ is in $\MAW(W')\bigtriangleup\MAW(W)$ and
  $b_1\alpha^{\ell} b_2$ is in $\MAW(W')\bigtriangleup\MAW(W)$ for every $0 \le \ell \le k-2$
  since $b_1$ is absent from $W'$.
  Also, $\beta$ is in $\MAW(W')\bigtriangleup\MAW(W)$ and
  $b_{\pi-1} \alpha^{\ell} \beta$ is in $\MAW(W')\bigtriangleup\MAW(W)$
  for every $0 \le \ell \le k-2$.
  Thus, at least $(\pi-3) k + k + 1 + (k-1) = (\pi-1) k$ MAWs are in $\MAW(W')\bigtriangleup\MAW(W)$.
  Additionally, the number $\pi-1$ of distinct characters which occur in $W$
  and are not equal to $\alpha$
  is at least $\lfloor (\sigma-1)/2 \rfloor$, since
  $k\lfloor(\sigma-1)/2\rfloor \le k(\sigma-1)/2 = (k - k/2)(\sigma-1) \le (k - 1)(\sigma-1) \le d$ where the second inequality follows from $k \geq 2$．
  Therefore, $|\MAW(W')\bigtriangleup\MAW(W)| \ge (\pi-1) k \ge \lfloor(\sigma-1)/2\rfloor k \in \Omega(\sigma k) = \Omega(d)$.
  The number of pairs of two adjacent windows $W$ and $W'$
  where $|\MAW(W')\bigtriangleup\MAW(W)| \in \Omega(d)$
  is $\Theta((n-d)/k)$.
  Therefore, we obtain
  $\mathcal{S}(T', d) \in \Omega(d(n-d)/k) = \Omega(\sigma(n-d)) = \Omega(\sigma n)$
  since $n-d \in \Omega(n)$.
  \begin{figure}[tb]
      \centerline{\includegraphics[width=0.9\linewidth]{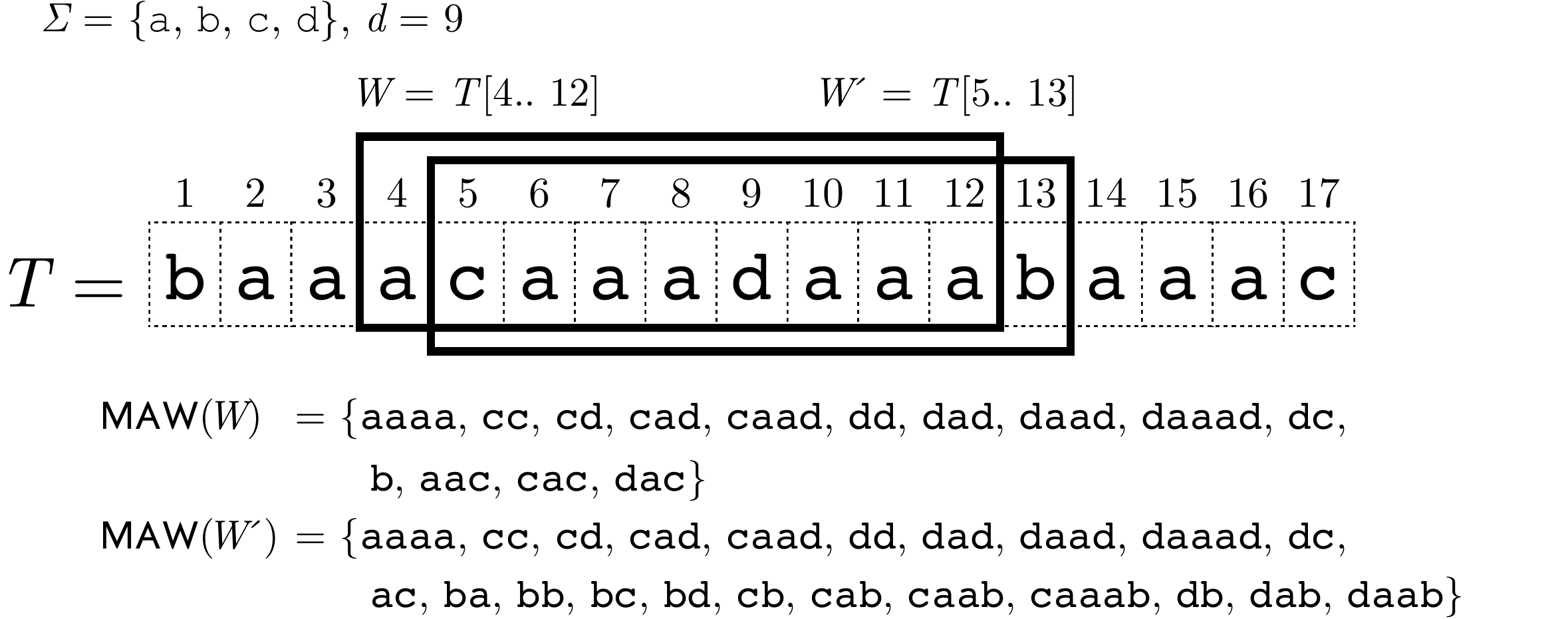}}
      \caption{
        Illustration of examples of MAWs for adjacent two windows.
        In this example, $\sigma = 4, d = 9 $, and $k = 4$.
        The size of the symmetric difference of $\MAW(W)$ and $\MAW(W')$ is
        $|\MAW(W) \bigtriangleup \MAW(W')| =$
        $|\{\mathtt{b},$ $\mathtt{aac},$ $\mathtt{cac},$ $\mathtt{dac},$ $\mathtt{ac},$
        $\mathtt{ba},$ $\mathtt{bb},$ $\mathtt{bc},$ $\mathtt{bd},$
        $\mathtt{cb},$ $\mathtt{cab},$ $\mathtt{caab},$ $\mathtt{caaab},$
        $\mathtt{db},$ $\mathtt{dab},$ $\mathtt{daab}\}| = 16$.
      }\label{fig:MAW_total}
  \end{figure}
  %
%
%  Finally, we assume that $n-d \in o(d)$.
%  In this condition, we consider a string $T'' = a_1 a_2 a_1^{d-2} a_2 a_1^{n-d-1}$.
%  On the first slide, $T''[d+1] = a_2$ is added to the window.
%  Then, $a_2 a_1^{\ell} a_2$ for every $1 \le \ell \le d-1$
%  are in $\MAW(T[1..d]\bigtriangleup\MAW(T[2..d+1])$.
%  Therefore, $\mathcal{S}(T', d) \in \Omega(d)$ even if $n-d \in o(d)$.
\end{proof}

Next, we consider the case where $\sigma \ge d+1$.
\begin{lemma}\label{lem:d_slide_MAW}
  For a string $T$ of length $n > d$ over an alphabet $\Sigma$ of size $\sigma$,
  $\mathcal{S}(T, d) \in O(d(n-d))$,
  and this upper bound is tight when $\sigma \ge d+1$.
\end{lemma}

\begin{proof}
  By Theorem~\ref{col:increase_MAW_at_one_slide},
  it is clear that $\mathcal{S}(T, d) \in O(d(n-d))$.
  Next, we show that there is a string $T'$ of length $n > d$ such that
  $\mathcal{S}(T', d) \in \Omega(d(n-d))$
  for any integer $d$ with $1 \le d \le \sigma-1$.
  Let $\Sigma = \{a_1, a_2, \cdots, a_\sigma\}$.
  We consider a string $T' = (a_1 a_2 \cdots a_{d+1})^e a_1 a_2 \cdots a_k$
  where $e = \lfloor n/(d+1) \rfloor$ and $k = n\bmod (d+1)$.
  For each window $W = T'[i..i+d-1]$ in $T'$,
  $W$ consists of distinct $d$ characters,
  and the character $T'[i+d]$ that is the right neighbor of $W$
  is different from any characters that occur in $W$.
  Without loss of generality , we assume that the current window is
  $W = a_1 a_2 \cdots a_d$
  and the next window is $W' = W[2..]a_{d+1} = a_{2} \cdots a_{d+1}$.
  Then, $|\MAW(W') \bigtriangleup \MAW(W)| = \
  |\{{a_1}^2 , a_{d+1}, a_2a_1, \ldots , a_da_1, a_1a_3, \ldots , a_1a_d \} \cup \{a_1, {a_{d+1}}^2 , a_{d+1}a_2, \ldots , a_{d+1}{a_d},  a_2a_{d+1}, \ldots , a_{d-1}a_{d+1} \}| = 4d-2 \in \Omega(d)$.
  Therefore, $\mathcal{S}(T', d) = \Omega{(d(n-d))}$.
\end{proof}

The main result of this section follows from the above lemmas:
\begin{theorem}\label{thm:total_slide_MAW}
  For a string $T$ of length $n > d$ over an alphabet $\Sigma$ of size $\sigma$,
  $\mathcal{S}(T, d) \in O(\min\{d, \sigma\} n)$.
  This upper bound is tight when $n-d \in \Omega(n)$.
\end{theorem}
We remark that $n-d \in \Omega(n)$ covers most interesting cases
for the window length $d$, since the value of $d$ can range
from $O(1)$ to $cn$ for any $0 < c < 1$.

\section{Tighter bounds for binary alphabets}
In this section we consider the case where $\sigma' = 2$,
i.e. when both the current sliding window $S = T[i..i+d-1]$ and
the next window $S\alpha = T[i..i+d]$ extended
with a new character $\alpha = T[i+d]$
consist of two distinct characters.
The goal of this section is to show that
when $\sigma' = 2$,
there exists a tighter upper bound for the number of changes of
MAWs than the general case with $\sigma' \geq 3$.
%In case $\sigma = 2$, The Upper bound of the number of changes of MAWs when appending a character to $T$ is smaller than that of the case of $\sigma \geq 3$.

%We consider two cases about MAW $aub$ : (1) $u=\varepsilon$, (2)$u \neq \varepsilon$. 
In what follows, let us denote by $\BSigma = \{ \mathtt{0}, \mathtt{1} \}$ the binary alphabet,
and assume without loss of generality that we append the new character $\alpha = \mathtt{0}$
to the window $S$ of length $d$ and obtain the extended window $S \alpha = S\mathtt{0}$.
%We assume that two characters are 0 and 1, and the character to be appended is 0, to make it simple. 

As a warm up, we begin with the two following lemmas which show that
at most 3 MAWs can change in the cases where $d = 1$ and $d = 2$
for any binary strings.

\begin{lemma} \label{lem:case_d=1}
  For any string $S$ over $\BSigma$ with $|S| = d = 1$,
  $|\MAW(S) \bigtriangleup \MAW(S\mathtt{0})| \leq 3$.
\end{lemma}

\begin{proof}
  For each $S \in \{\mathtt{0}, \mathtt{1}\}$ of length $1$,
  \begin{eqnarray*}
    \MAW(\mathtt{0}) \bigtriangleup \MAW(\mathtt{00}) & = & \{\underline{\mathtt{00}}, \mathtt{000}\}, \\
    \MAW(\mathtt{1}) \bigtriangleup \MAW(\mathtt{10}) & = & \{\underline{\mathtt{0}}, \mathtt{00}, \mathtt{01}\}, 
  \end{eqnarray*}
  where the underlined strings are those in $\MAW(S) \setminus \MAW(S\mathtt{0})$
  and the strings without underlines are those in $\MAW(S\mathtt{0}) \setminus \MAW(S)$.
  Thus the lemma holds.
\end{proof}

\begin{lemma} \label{lem:case_d=2}
  For any string $S$ over $\BSigma$ with $|S| = d = 2$,
  $|\MAW(S) \bigtriangleup \MAW(S\mathtt{0})| \leq 3$.
\end{lemma}

\begin{proof}
  For each $S \in \{ \mathtt{00}, \mathtt{01}, \mathtt{10}, \mathtt{11}\}$ of length $2$,
  \begin{eqnarray*}
    \MAW(\mathtt{00}) \bigtriangleup \MAW(\mathtt{000}) & = & \{\underline{\mathtt{000}}, \mathtt{0000}\}, \\
    \MAW(\mathtt{01}) \bigtriangleup \MAW(\mathtt{010}) & = & \{\underline{\mathtt{10}}, \mathtt{101}\}, \\
    \MAW(\mathtt{10}) \bigtriangleup \MAW(\mathtt{100}) & = & \{\underline{\mathtt{00}}, \mathtt{000}\}, \\
    \MAW(\mathtt{11}) \bigtriangleup \MAW(\mathtt{110}) & = & \{\underline{\mathtt{0}}, \mathtt{00}, \mathtt{01} \},
  \end{eqnarray*}
  where the underlined strings are those in $\MAW(S) \setminus \MAW(S\mathtt{0})$
  and the strings without underlines are those in $\MAW(S\mathtt{0}) \setminus \MAW(S)$.
  Thus the lemma holds.
\end{proof}

We move onto the case where $d \geq 3$.
Our first observation is that %we can restrict ourselves to the case where $S$ is not unary.
it is sufficient to consider the case that $S$ is not unary.
For any $d$, it is clear that
$|\MAW(\mathtt{0}^{d}) \bigtriangleup \MAW(\mathtt{0}^{d+1})| = 2$.
Now let us consider $\mathtt{1}^d$ in the next lemma.

\begin{lemma} \label{lem:at_least_a_zero}
  For any $d \geq 3$ let $V = \mathtt{1}^d$.
  Then, there exists another string $S$ of length $d$ over $\BSigma$
  such that $S[k] = \mathtt{0}$ for some $1 \leq k \leq d$
  and $|\MAW(V) \bigtriangleup \MAW(V\mathtt{0})| \leq |\MAW(S) \bigtriangleup \MAW(S\mathtt{0})|$.
\end{lemma}

\begin{proof}
  Since $V = \mathtt{1}^d$,
  $\MAW(V) \setminus \MAW(V\mathtt{0}) = \{ \mathtt{0} \}$.
  Also, $\MAW(V\mathtt{0}) \setminus \MAW(V) = \{ \mathtt{00}, \mathtt{01} \}$.
  Thus $|\MAW(V) \bigtriangleup \MAW(V\mathtt{0} )| = 3$ for any $d \geq 1$.
%  $0 \in \MAW(V)$ but $0 \notin \MAW(V0)$.
%  Also, $00, 01 \notin \MAW(V)$ but $00, 01 \in \MAW(V0)$.
%  It is clear that these are the only MAWs that change
%  between $V$ and $V0$, namely $|\MAW(V) \bigtriangleup \MAW(V0)| = 3$
%  for any $d \geq 1$.

  Let $S = \mathtt{01}^{d-1}$ and $S\mathtt{0} = \mathtt{01}^{d-1}\mathtt{0}$ with $d \geq 3$.
  Then,
  $\MAW(S\mathtt{0}) \setminus \MAW(S) = \{ \mathtt{01}^k\mathtt{0} \mid 1 \leq k \leq d-2\} \cup \{ \mathtt{101} \}$
  and $\MAW(S\mathtt{0}) \setminus \MAW(S) = \{ \mathtt{10} \}$.
%  $010$, \ldots, $01^{d-2}0, 101 \notin \MAW(W)$ but
%  $010$, \ldots, $01^{d-2}0 \in \MAW(W0)$.
%  Also, $10 \in \MAW(W)$ but $10 \notin \MAW(W0)$.
  Thus we have $|\MAW(S) \bigtriangleup \MAW(S\mathtt{0})| \geq d \geq 3$.
\end{proof}

According to Lemmas~\ref{lem:case_d=1}, ~\ref{lem:case_d=2} and ~\ref{lem:at_least_a_zero},
in what follows we focus on the case where
$d \geq 3$ and the current window $S = T[i..i+d-1]$ contains at least one $\mathtt{0}$.
The latter condition implies that
we focus on the case where the new character $\alpha = \mathtt{0}$
already occurs in the window $S$.

As in the case of non-binary alphabets,
we analyze the numbers of added Type-1/Type-2/Type-3 MAWs 
in $\typeone$/$\typetwo$/$\typethree$ for binary strings.
Recall that in the current context, for any $S = T[i..i+d-1]$,
a MAW $w$  in $\MAW(S\mathtt{0}) \setminus \MAW(S)$ is said to be of:
\begin{itemize}
  \item Type 1 if neither $w[2..]$ nor $w[..|w|-1]$ occurs in $S$;
  \item Type 2 if $w[2..]$ occurs in $S$ but $w[..|w|-1]$ does not occur in $S$;
  \item Type 3 if $w[2..]$ is does not occur in $S$ but $w[..|w|-1]$ occurs in $S$.
\end{itemize}

We first show the upper bound for the size of $\typethree$
in the case where $\sigma' = 2$.

\begin{lemma}\label{lem:bi_type3}
  For any binary string $S$ over $\BSigma$
  such that $|S| = d \geq 3$, $|\typethree| \leq d-2$.
\end{lemma}

\begin{proof}
Recall the proof for Lemma~\ref{type3max}.
There, we proved that each MAW $w$ of Type 3 for
any non-binary string $R\alpha = T[i..i+d] = T[i..j+1]$
is mapped by an injection $f$ to a distinct position of $T[i..j]$
in the range $[i, j-1]$,
or alternatively to a distinct position of $R$ in range $[1, d-1]$.
This showed $|\typethree| \leq d-1$ for $\sigma' \geq 3$.

Here we show that the range of such an injection $f$ is $[2, d-1]$
for any binary string $S$ with $\sigma' = 2$.
Since the appended character is $\alpha = \mathtt{0}$,
and since the candidate $x$ for the MAW of Type 3 which should be mapped to
the first position in $S$ is of length $2$,
the candidate $x$ has to be either $\mathtt{00}$ or $\mathtt{10}$.
\begin{enumerate}
\item[(1)] If $x = \mathtt{00}$, then $S[1] = \mathtt{0}$.
  If $\mathtt{00}$ does not occur in $S$ (see also the top picture of Figure~\ref{fig:binarymaw_figure1}),
  then $\mathtt{00}$ is already a MAW for $S$
  (i.e. $\mathtt{00} \in \MAW(S)$).
  Thus $\mathtt{00} \notin \MAW(S\mathtt{0}) \setminus \MAW(S)$ in this case.
  Otherwise ($\mathtt{00}$ occurs in $S$), then clearly $\mathtt{00}$ is not a MAW for $S\mathtt{0}$
  (see also the middle picture of Figure~\ref{fig:binarymaw_figure1}).
  \item[(2)] If $x = \mathtt{10}$, then $S[1] = \mathtt{1}$.
    However, since the appended character is $\mathtt{0}$,
    $\mathtt{10}$ must occur somewhere in $S\mathtt{0}$
    (see also the bottom picture of Figure~\ref{fig:binarymaw_figure1}).
    Thus $\mathtt{10}$ is not a MAW for $S\mathtt{0}$.
\end{enumerate}
Hence, the first position of $S$ cannot be
assigned to any MAW of Type 3 for $S\mathtt{0}$,
leading to $|\typethree| \leq d-2$ for any binary string $S$ of length $d \geq 3$.
%In $\sigma' = 2$ ,in addition to lemma \ref{type3max} in the case $\sigma \geq 3$, $W[1]$ cannot be the end position of $aub$ of type 3. if it can be the position, it is 00($W[1] = 0$) or 10($W[1] = 1$). If 00, $W$ doesn't have the substring 00, since it is MAW. However, it is also absent before appending 0 since $W[1] = 0$ (Figure \ref{binarymaw_figure1}). If 10, that means $W[1] =1$ and $W[d+1] = 0$, is cannot be MAW, it exists in $W0$ definitiely. Therefore, $W[1]$ cannot be the end position of $aub$ of type 3. It means that the new MAW $\typethree$ is at most $(d-1) -1 = d-2$.
\end{proof}

In other words, Lemma~\ref{lem:bi_type3} shows that in the binary case
with $\sigma' = 2$,
the maximum number of added Type-3 MAWs is 1 less than in the case
with $\sigma' \geq 3$.

\begin{figure}[tb]
 \centering
 \includegraphics[keepaspectratio, scale=0.35]
      {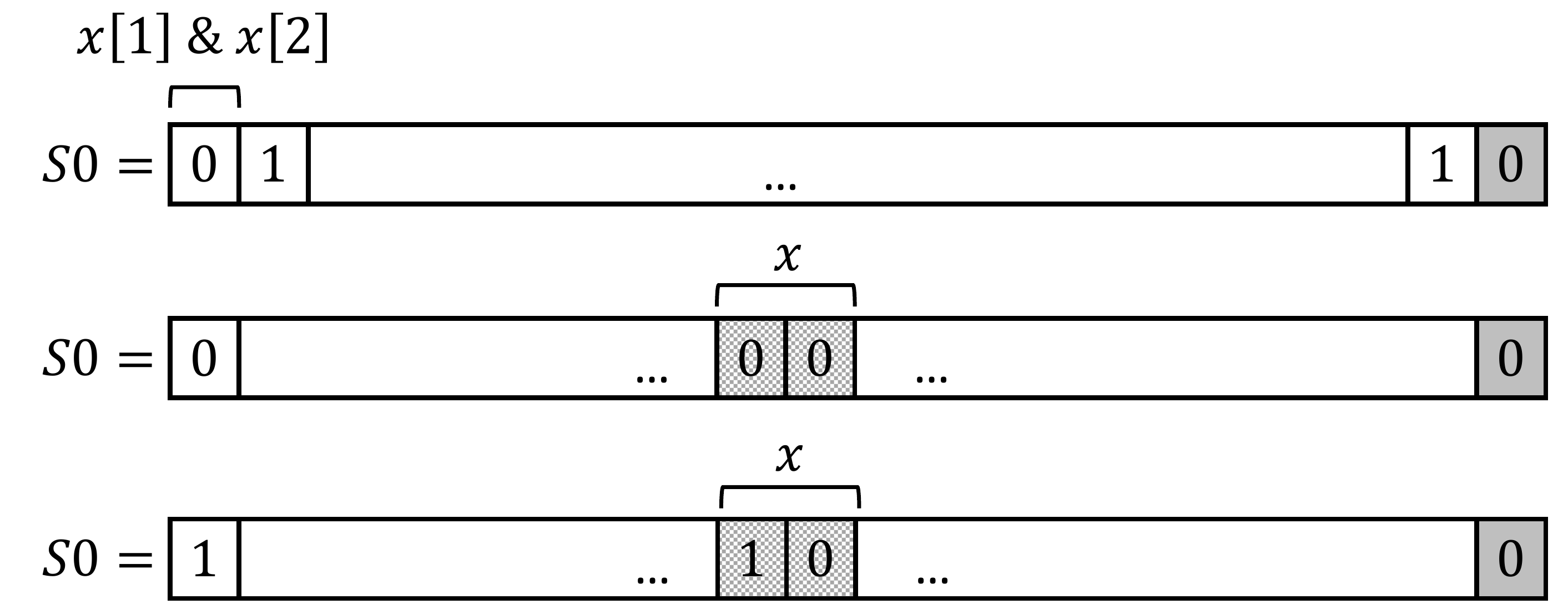}
      \caption{Characteristics of Type-3 MAWs in the binary case with $\sigma' = 2$,
      where the rightmost $\mathtt{0}$ in gray is the new appended character in each picture.}
 \label{fig:binarymaw_figure1}
\end{figure}

  Next, we consider the total number of added Type-1/Type-2 MAWs.
  
From Lemma~\ref{MAW1and2_collide}, the next corollary holds. 
  \begin{corollary}
  \label{MAW1and2_binary_collide}
  $|\typeone|+|\typetwo| \leq 2$ on any binary string $S$ when the character to be appended already occurs in $S$.
  \end{corollary}
  
%The next lemma holds for general alphabets of arbitrary size: 
%  Focusing on the total upper bounds of $\typeone \typetwo$, Following lemma holds :
%denote $k$ by the number of 0 that occurs in $W$ as suffix. if $T[|T|] \neq 0 $, $k=0$. $\typeone$ MAW can be generated is $0^{k+2}$ ($au=ub=0^{k+1}$ first appears in $W0$).
%For increasing $\typeone$ MAW for $W0$, $W$ cannot have any consecutive 0 that longer than $k$.
%If not, $au=ub(= 0^{k+1})$ is already in $W$. Therefore, $\typeone$ MAW cannot be generated.
%On the other hand, if $\typeone$ MAW increases, $\typetwo$ MAW that ends 0 cannot be MAW for $W$.   If it is longer than $k$ denote $a'u'0$ by $\typetwo$ MAW that ends 0. $ub$ has $0^{k+1}$ as their Suffix, and this leads to make a contradiction that $W$ cannot have any consecutive 0 that longer than $k$(Figure \ref{binarymaw_figure2}). Therefore, if the character to be appended is already in $W$, $\typeone \typetwo$ that increases in $W\alpha$ never reach 3 in total. 

%\begin{figure}[tb] move to maw.tex
 %\centering
 %\includegraphics[keepaspectratio, scale=0.35]
 %     {img/cp2.pdf}
 %\caption{Collision between the new Type-2 MAW and Type-1 MAW in the binary case, where the rightmost $\mathtt{0}$ in gray is the new appended character in each picture.}
% \label{binarymaw_figure2}
%\end{figure}

%Now there are two independent upper bounds in every type of MAW.
%the relatonships are summarized in table ~\ref{tab:joukaimatome}.

A direct consequence of Lemma~\ref{lem:bi_type3} and Lemma \ref{MAW1and2_collide}
is an upper bound for the added MAWs
$|\typeone|+|\typetwo|+|\typethree| \leq d$ for any binary string $S0$
with $|S| = d \geq 3$.
In what follows, we further reduce this upper bound to $|\typeone|+|\typetwo|+|\typethree| \leq d-1$. For this purpose, we introduce the next lemma:

\begin{lemma}\label{lem:bi_newupper}
  For any binary string $S$ over $\BSigma$
  such that $|S| = d \geq 3$,
  $|\typetwo|$ is at most the number of occurrences of $\mathtt{0}$ in $S[1..d-1]$,
  and $|\typethree|$ is at most the number of occurrences of $\mathtt{1}$ in $S[3..d]$.
%
%In binary case, when appending 0 in string $W$, the number of new $\MAW \typetwo$ is lower than the number of 0 in $W[1..d-1]$, and the that one of $\typethree$ is also lower than the number of 1 in $W[3..d]$. 
\end{lemma}

\begin{proof}
First we consider Type-2 MAWs for $S\mathtt{0}$.
Since $S$ is a binary string, by Lemma \ref{type2max},
there are at most two MAWs in $\typetwo$.
We assume that there are two MAWs in $\typetwo$ and let $au\mathtt{0}$ and $a'\mathtt{1}$ be the two MAWs where $a, a' \in \BSigma$ and $u, u' \in \BSigma^*$.
By the definition of Type-2 MAWs, $u\mathtt{0}$ and $u'\mathtt{1}$ occur in $S$ and $u[|u|] = u'[u'] = \mathtt{0}$ since $au$ and $a'u'$ occur in $S\mathtt{0}$ as suffixes.
For any two positions $t_0$ and $t_1$ such that $S[t_0 .. t_0+|u\mathtt{0}|-1]=u\mathtt{0}$, $S[t_1 .. t_1+|u' \mathtt{1}|-1]=u'\mathtt{1}$, $t_0+|u| \neq t_1 + |u'|$. Consequently, there are at least two occurrences of $\mathtt{0}$ in $S$ since $S[t_0 + |u| -1] = \mathtt{0}$ and $S[t_1 + |u'| -1] = \mathtt{0}$.
Since $t_0 + |u\mathtt{0}| - 1 \leq d$, $|\typetwo|$ is at most the number of occurrences of $\mathtt{0}$ in $S[1..d-1]$.
% Let $aub$ denote a MAW in $\typetwo$,
%  where $a, b \in \BSigma$ and $u \in \BSigma^*$.
 % First we consider Type-2 MAWs for $S\mathtt{0}$.
  %By the definition of $\typetwo$,
  %the last character of $au$ is $0$ since $au$ occurs in $S0$ as the suffix.
  %Consider two distinct Type-2 MAWs $aub$ and $a'u'b$
  %where $a, a', b \in \BSigma$, $u, u' \in \BSigma^*$, and
  %$|aub| < |a'u'b|$.
  %Since both $aub$ and $a'u'b$ are Type-2 MAWs for $S0$,
  %both $au$ and $a'u'$ are suffixes of $S$.
  %In addition, $au$ is a proper suffix of $a'u'$ since $|aub| < |a'u'b|$.
  %This implies that 
  %$ub$ and $u'b$ cannot have occurrences in $S$
  %with the same ending positions, since
  %otherwise $aub$ must occur in $S$, a contradiction.
  %Since the last characters of $au$ and $a'u'$ are both $0$,
  %they cannot share the same ending positions in $S$.
  %In addition, the last position $d = |S|$ in $S$
  %cannot be the ending position of $au$ for any Type-2 MAW $aub$
  %since $aub$ does not occur in $S0$.
  %Thus, the total number of Type-2 MAWs for $S0$
  %is upper bounded by the number of $0$'s in $S[1..d-1]$.
  
  Second we consider Type-3 MAWs for $S\mathtt{0}$.
  let $au\mathtt{0}$ be the Type-3 MAW where $a \in \BSigma$, $u, \in \BSigma^*$ since $u\mathtt{0}$ must be a suffix of $S\mathtt{0}$.
  By the definition of Type-3 MAW,
  there has to be an occurrence of $au$ in $S$.
  Note that this occurrence has to be immediately followed by a $\mathtt{1}$
  since $au\mathtt{0}$ does not occur in $S\mathtt{0}$. 
  Thus, for each $au\mathtt{0} \in \typethree$,
  we need an occurrence of $au\mathtt{1}$ in $S$.
  Since $|au| \geq 1$, we clearly cannot use
  the first position of $S$ as the ending position of $au1$.
  Also, it follows from Lemma~\ref{lem:bi_type3} (and its proof) that
  the second position of $S$ cannot be the ending position of $au$
  for any Type-3 MAW $aub$ for $S\mathtt{0}$.
  This implies that there is no Type-3 MAW that corresponds to
  the $\mathtt{1}$ in the second position of $S$.
  Thus, the total number of Type-3 MAWs for $S\mathtt{0}$
  is upper bounded by the number of occurrences of $\mathtt{1}$ in $S[3..d]$.
%consider new MAW $\typethree$ in the form of $aub$.
%  For $\typetwo$, when we denote $\typetwo$ MAW by $aub$, $u[|u|] = 0$ is essentical since $au$ doesn't appears $W$ but appears in $W0$. Any two MAW $\typetwo$ cannot have the same end position because each $\typetwo$ MAW have $0$ in their $au$. Considering this and lemma \ref{type2max}, the number of MAW $\typetwo$ is at most per $0$ in $W$. Additionally, $W[d]$ never be an end position of MAW $\typetwo$ since $aub$ cannot be in $W0$. Therefore, $\typetwo$ is less than the number of $0$, the same letter of the character to be appended in $W[1...d-1]$.
%
%  Next, we focus on the number of MAW $\typethree$. All MAW in this group, can be denoted by $au0$, since $ub$ is the suffix of $W0$.
%  However $au$ must appear in $W$, and 1 must follow $au$ because $au0$ is not in $W$. Similar to the case of $\typetwo$, any two MAW $\typethree$ never collide at their end position of $au$. Also, 1 in $W[1]$ and $W[2]$, does not have a corresponding MAW (In case $W[2]$, it is by the property that There are no $\typethree$ MAW $a0$ that $a$ appears in $W[1]$, by lemma \ref{lem:bi_type3}). Therefore, the number of MAW $\typethree$ is less than the number of $1$, the different letter from the character to be appended in $W[3...d]$.
\end{proof}

Intuitively, Lemma~\ref{lem:bi_newupper} implies that
flipping substrings $\mathtt{1}$ in $S[3..d-1]$
does not increase the total number of Type-2 and Type-3 MAWs for $S\mathtt{0}$.

\begin{lemma} \label{lem:tight_binary_added_bound}
  For any binary string $S$ over $\BSigma$ with $|S| = d \geq 3$,
  $|\typeone|+|\typetwo|+|\typethree| \leq d-1$.
\end{lemma}

\begin{proof}
  \tanote*{made the proof easy}{
  $S$ is not unary due to Lemma \ref{lem:at_least_a_zero}.
    %  Let us try to construct another string $S'$ over $\BSigma$ such that $|\typeone|+|\typetwo|+|\typethree| = d$.
  It immediately follows from Lemma~\ref{lem:bi_type3} and Corollary~\ref{MAW1and2_binary_collide} that $|\typeone|+|\typetwo|+|\typethree| \leq d$,
  and assume on the contrary that there exists a binary string $S'$ over $\BSigma$ such that $|\typeone|+|\typetwo|+|\typethree| = d$.
  Then, it has to be $|\typeone| + |\typetwo| = 2$ and $|\typethree| = d-2$, again by Lemma~\ref{lem:bi_type3} and Corollary~\ref{MAW1and2_binary_collide}.  %holds by Lemma \ref{MAW1and2_collide} and \ref{lem:bi_type3}, and both of inequalities must be equalities.
  Therefore, $S'[3..d]  = \mathtt{1}^{d-2}$ by Lemma \ref{lem:bi_newupper}, and $|\typeone| = 0$. Hence we must have $|\typetwo| = 2$, which leads to $S' = \mathtt{00}\mathtt{1}^{d-2}$.
  }
  %Let $U$ be a binary string over $\BSigma$ 
  %such that $U0$ has the maximum total number of added MAWs.
  %It follows from Lemma~\ref{lem:bi_type3} and Lemma~\ref{lem:bi_newupper}
  %that $U[3..d] = 1^{d-2}$.
  %Recall that by Lemma~\ref{type1max}
  %Type-1 MAW must be of form $0^h$
  %and $0^{h-2}$ has to be a suffix of $U$.
  %However, since $(U0)[d..d+1] = 10$
  %and $U$ contains at least a $0$ by Lemma~\ref{lem:at_least_a_zero},
  %the only candidate $0$ cannot be a Type-1 MAW for $U0$.
  %Thus there is no Type-1 MAW for $U0$.
  %Using Lemma~\ref{MAW1and2_collide},
  %we can now conclude that $U[1..2] = 00$,
  %and thus $U = 001^{d-2}$.
  Now the sets of all the added MAWs for $S'\mathtt{0} = \mathtt{00}\mathtt{1}^{d-2}\mathtt{0}$ are
  \begin{eqnarray*}
    \typeone & = & \emptyset, \\
    \typetwo & = & \{ \mathtt{100}, \mathtt{101} \}, \\
    \typethree & = & \{ \mathtt{0}\mathtt{1}^k\mathtt{0} \mid 1 \leq k \leq d-3\},
  \end{eqnarray*}
  which leads to $|\typeone|+|\typetwo|+|\typethree| = d-1$, a contradiction. Thus the lemma holds.
%  Since any other string $V\mathtt{0}$ over $\BSigma$ with $|V| = d$ has less added MAWs than $S'\mathtt{0}$,
%  the lemma holds.
\end{proof}

%Lemma~\ref{MAW1and2_collide} and Lemma~\ref{lem:bi_newupper} implies that the maximum MAW is potentially $d$, and it can be achieved only with (1) the string that $W[3..d] = 1..1$, (2) the string that has two 0.
%The only string that satisfies all of these conditions and has the potential to achieve $d$ MAW is this : 
%
%\begin{equation}
%W = 001...1
%\end{equation}
%
%This is only the string that satisfies that $W[3..d] = 1..1$ and two 0 in $W$ which is necessary to achieve the maximum number of MAW of Type 2 and 3.
%Therefore, the upper bound is not $d$ if $d$ MAW does not generated in the process of changing $W$ to $W\alpha$. 
%Actually, new MAWs generated in this process are \{ $010,0110,...01^{d-3}0,100, 101$ \},
%when d $\geq$ 3, and shows that $d$ MAW is impossible. 

%At the same time, this string is exactly an example of Lower bound. ($d-1$ MAW generated and a MAW(=10) appears in $T\alpha$). Considering that a MAW disappears in appending process, finally, The theorem mentioned below holds : 

We obtain our main theorem:

\begin{theorem}
\label{lem:bi_add_total}
For any binary string $S$ over $\BSigma$ with $|S| = d \geq 3$,
$|\MAW(S) \bigtriangleup \MAW(S\mathtt{0})| \leq d$,
and this upper bound is tight.
%Also, there is a unique binary string that achieves this bound.
%  When $d \geq 3$, the MAW changes when appending a character to a string of length $d)$ in binary case, is at most $d$, and it is tight.
\end{theorem}

\begin{proof}
The upper bound follows from Lemmas~\ref{lem:extention_dec} and~\ref{lem:tight_binary_added_bound},
and its tightness follows from
and our construction of the string $S'\mathtt{0}$ in the proof for Lemma~\ref{lem:tight_binary_added_bound}.
\end{proof}

The next corollary summarizes the results of this section.

\begin{corollary}
For any binary string $S$ over $\BSigma$ with $|S| = d$,
$|\MAW(S) \bigtriangleup \MAW(S\mathtt{0})| \leq \max\{3, d\}$,
and this upper bound is tight for any $d \geq 1$.
\end{corollary}

\begin{proof}
  The upper bound follows from Lemmas~\ref{lem:case_d=1} and~\ref{lem:case_d=2}, Theorem~\ref{lem:bi_add_total},
  and its tightness follows from all possible binary cases shown in the proof for Lemmas~\ref{lem:case_d=1} and~\ref{lem:case_d=2} for $d = \{1, 2\}$,
and our construction of the string $S'\mathtt{0}$ in the proof for Lemma~\ref{lem:tight_binary_added_bound} for $d \geq 3$.
\end{proof}

\section{Conclusions and future work}

In this paper, we revisited the problem of
computing the minimal absent words (MAWs) for the sliding window model,
which was first considered by Crochemore et al.~\cite{CrochemoreHKMPR20}.

We investigated combinatorial properties of
MAWs for a sliding window of fixed length $d$ over a string of length $n$.
Our contributions are \emph{matching upper and lower bounds}
for the number of changes in the set of MAWs
for a sliding window when the window is shifted to the right by one character.
For the general case where the window $S$ and the extended window $S\alpha$
contain three or more distinct characters (i.e. $\sigma' \geq 3$),
the number of changes in the set of MAWs for $S$ and $S\alpha$ is
at most $d+\sigma'+1$ and this bound is tight.
For the case of binary alphabets (i.e. $\sigma' = 2$),
it is upper bounded by $\max\{3, d\}$ and this bound is also tight.

We also gave an asymptotically tight bound $O(\min\{d, \sigma\}n)$
for the number $\mathcal{S}(T, d)$ of total changes in the set of MAWs
for every sliding window of length $d$ over any string $T$ of length $n$,
where $\sigma$ is the alphabet size for the whole input string $T$.

The following open questions are intriguing:
\begin{itemize}
\item We showed that a matching lower bound $\mathcal{S}(T, d) \in \Omega(\min\{d, \sigma\}n)$ when $n-d \in \Omega(n)$. Is there a similar lower bound when $n-d \in o(n)$?

\item Crochemore et al.~\cite{CrochemoreHKMPR20}
  gave an online algorithm that maintains the set of MAWs
  for a sliding window of length $d$ in $O(\sigma n)$ time.
  Can one improve the running time to optimal $O(\min\{d, \sigma\}n)$?
\end{itemize}

\section*{Acknowledgments}
This work was supported by JSPS KAKENHI Grant Numbers JP20J11983 (TM), JP18K18002 (YN), JP21K17705 (YN), JP17H01697 (SI), JP16H02783 (HB), JP20H04141 (HB), JP18H04098 (MT), and by JST PRESTO Grant Number JPMJPR1922 (SI).

The authors thank anonymous referees for their fruitful comments and suggestions for improving the presentation of the paper.

\bibliographystyle{abbrv}
\bibliography{ref}

\end{document}